\documentclass[a4paper,11pt]{amsart}
\usepackage{amssymb}
\author[A.~P.~Petravchuk]{Anatoliy P. Petravchuk}
\address{
Department of Algebra and Mathematical Logic\\ Faculty of Mechanics and Mathematics\\
Kyiv Taras Shevchenko University\\ Volodymyrska street, 01033
Kyiv, Ukraine} \email{aptr@univ.kiev.ua}
\title[On behavior of solvable ideals of Lie algebras under outer derivations]
{On behavior of solvable ideals of Lie algebras under outer
derivations}
\thanks{Supported by DFFD, Grant F25.1/095 }

\newtheorem{theorem}{Theorem}

\newtheorem{lemma}{Lemma}
\theoremstyle{definition}

\theoremstyle{remark}

\theoremstyle{problem}

\def\char{{\rm char\ }}

\let\leq\leqslant
\let\geq\geqslant
 
 \def\char{\rm char\,}
\begin{document}

\sloppy

\begin{abstract}
Let $L$ be a  finite dimensional Lie algebra  over a field $F$. It
is well known that the solvable radical $S(L)$ of the algebra $L$
is a characteristic ideal of $L$ if $\char F=0$ and there are
counterexamples to this statement in case  $\char F=p>0$. We prove
that the sum $S(L)$ of all solvable ideals of a Lie algebra $L$
(not necessarily finite dimensional) is a characteristic ideal of
$L$ in the following cases: 1) $\char F=0;$ 2) $S(L)$ is solvable
and its derived length is less than $\log _{2}p.$ Some estimations
(in characteristic $0$) for the derived length of ideals
$I+D(I)+\cdots +D^{k}(I)$ are obtained where $I$ is a solvable
ideal of $L$ and $D\in Der(L).$

\end{abstract}

\maketitle

\section{Introduction}
\label{} Let $L$ be a finite dimensional Lie algebra over a field
of characteristic $0$. It is  well-known  that its solvable
radical $S(L)$ is a characteristic ideal of $L,$ i.e.
$D(S(L))\subseteq S(L)$ for every derivation $D\in Der (L)$ (see,
for example, \cite{Jacob}, Ch.III, Th.6.7). This result breaks
down in characteristic $p>0$ (see \cite{Jacob}, p.74-75). The
noted counter-example has solvable radical of derived length $[ \
\log _{2}p]+1$ (because it is the tensor product of a simple Lie
algebra over a field $F$ of characteristic $p$ and the group
algebra $F[G]$ where $G$ is the cyclic group of order $p$). We
prove that the solvable radical $S(L)$ is a characteristic ideal
if its derived length is less that $\log _{2}p$. The same result
is true for infinite dimensional Lie algebras if we denote by
$S(L)$ the sum of all solvable ideals of $L$. This sum can be
considered as an analogue of solvable radical of the algebra $L$
(the ideal $S(L)$ is locally solvable while the sum of all locally
solvable ideals of $L$ can be not locally solvable \cite{Lat}). If
$L$ is an (infinite dimensional) Lie algebra over a field of
characteristic $0$ we prove that $S(L)$ is always an
characteristic ideal of the algebra $L.$

Derivations  of Lie algebras and associative algebras  were
studied by many authors (see for example \cite{Seligman},
\cite{Hartley}, \cite{AmSt}). In particular, the behavior of
locally nilpotent ideals of Lie algebras (in characteristic $0$)
under derivations was studied in \cite{Hartley}, where a passage
from a derivation to an automorphism  of a Lie algebra was used.

The main results of this paper can be easily deduced from
Theorem~\ref{solvability}, where
  behavior of solvable ideals of  Lie algebras (not
necessarily finite dimensional) is studied under outer derivations
of $L.$ Let $I$ be a  solvable ideal of derived length $n$ of a
Lie algebra $L$ and let $D$ be a derivation $D\in Der(L)$. It is
proved that  the ideal $I+D(I)$ is solvable of derived length
$\leq 2n$ if $\char F=0$ or $n<\log _{2}p$ where $p=\char F$

  Ideals of the form  $I_{k}=I+D(I)+\cdots
+D^{k}(I)$ were also studied where $I$ is a solvable  ideal of
derived length $n$ of the algebra $L$  and $D\in Der(L).$ The
previous results give an estimation for the derived length of
$I_{k}$ of the form $s(I_{k})\leq 2^{k-1}n.$ But the
Theorem~{\ref{degree} shows that the growth of the derived length
is in fact polynomial, namely $s(I_{k})\leq f(k)$ for some
polynomial $f(x)$ of degree $n.$

The notations in the paper are standard. If $L$ is a Lie algebra
over a field $F$ then by $S(L)$ we shall denote the sum of all
solvable ideals of $L.$  We denote by $s(L)$ the derived length of
a solvable Lie algebra $L$. The $k$-th member of the derived
series of a Lie algebra $L$ is denoted by $L^{(k)}$, in
particular, $L^{(0)}=L, \ L^{(1)}=[L, L], \ L^{(n)}=[L^{(n-1)},
L^{(n-1)}]. $  In this notations $s(L)=0$ if and only if $L=0.$ If
$X$ is an $F$-subspace of a Lie algebra $L$ then $X^{1}=X,
X^{2}=[X, X], X^{n}=[X^{n-1}, X]$ are subspaces of $L$. If $D$ is
a derivation of a Lie algebra $L$ and $A$ is a $F$-subspace of $L$
then we denote for convenience $D^{0}(A)=A, \
D^{k}(A)=D(D^{k-1}(A))$ for $k\geq 1.$

 \section{On  characteristicity of the solvable radical of a Lie algebra}
\label{}

In the next Lemma we collect for convenience some known facts
which will be often used in the sequel.
\begin{lemma}\label{rule}

Let $L$ be a Lie algebra over an arbitrary field $F$, let $D$ be
its derivation. Then

1) for any ideal $I$ of the algebra $L$ the $F$-subspace
$I+D(I)+\cdots +D^{n}(I)$ is an ideal of $L$ for every $n=1, 2,
\ldots $;

2) for any elements $x, y\in L$ it holds the Leibnitz's rule:
$$D^{k}([x, y])=\sum _{s=0}^{k}{k\choose s}[D^{s}(x),
D^{k-s}(y)].$$
\end{lemma}

\begin{lemma}\label{deriv}

Let $L$ be a Lie algebra over an arbitrary field,  let $I$ be an
ideal of the algebra $L.$ Then for any derivation $D\in Der(L)$ it
holds $D^{m}(I^{(s)})\subseteq I$ for every $m\leq 2^{s}-1.$
\end{lemma}
\begin{proof}
Induction on $s.$ If $s=1,$ then $$D(I^{(1)})=D([I, I])\subseteq
[D(I), I]+[I, D(I)]\subseteq I$$ because $I$ is an ideal of $L.$
Let the statement of Lemma be true for $s-1$, let us prove it for
$s.$ Take  an arbitrary positive integer $m$ such that $m\leq
2^{s}-1.$ Then for arbitrary elements $x, y\in I^{(s-1)}$ it holds
$$ D^{m}([x, y])=\sum _{i=0}^{m}{m\choose i}[D^{i}(x),
D^{m-i}(y)]. \eqno (1)
$$ Since $i+(m-i)=m\leq 2^{s}-1$, then at least one
of the numbers $i$ or $m-i$ does not exceed $2^{s-1}-1.$ As
elements $x, y$ lie in $I^{(s-1)}$, we have by the induction
hypothesis that either $D^{i}(x)$ or $D^{m-i}(y)$ belong to the
ideal $I.$ But then from (1) it follows that $D^{m}([x, y])\in I.$
Thus $D^{m}(I^{(s)})\subseteq I$ for every nonnegative integers
$m\leq 2^{s}-1.$ The Lemma is proved.

\end{proof}

\begin{lemma}\label{key}

Let $L$ be a Lie algebra over a field $F$ and $I$ be  an ideal of
$L.$ If $char F$ does not divide the binomial coefficient
${2^{k}\choose 2^{k-1}}$, then for every $D\in Der(L)$ it holds
$$[D^{2^{k-1}}(I^{(k-1)}), D^{2^{k-1}}(I^{(k-1)})]\subseteq
D^{2^{k}}(I^{(k)})+I.$$

\end{lemma}

\begin{proof}

Let $x, y\in I^{(k-1)}$ be arbitrary elements. Then
$$D^{2^{k}}([x, y])=\sum _{s=0}^{2^{k}}{2^{k}\choose s}[D^{s}(x),
D^{2^{k}-s}(y)].$$ Rewrite the last sum as a sum of two summands:
$$ D^{2^{k}}([x, y])={2^{k}\choose 2^{k-1}}[D^{2^{k-1}}(x),
D^{2^{k-1}}(y)]+ $$ $$+\sum _{s=0, s\not=
2^{k-1}}^{2^{k}}{2^{k}\choose s}[D^{s}(x), D^{2^{k}-s}(y)] \eqno
(2) $$

In  the second summand of the  righthand side  of this formula
either the number $s$ or the number $2^{k}-s$ does not exceed
$2^{k-1}-1$ and therefore by Lemma~\ref{deriv} the sum $\sum
_{s=0, s\not= 2^{k-1}}^{2^{k}}{2^{k}\choose s}[D^{s}(x),
D^{2^{k}-s}(y)]$ lies in the ideal $I.$ Further, by conditions of
Lemma $\char F$ does not divide the binomial coefficient
$2^{k}\choose 2^{k-1}$, so from the equality $(2)$, taking into
account the arbitrary choice of $x$ and $y$, we obtain
$$[D^{2^{k-1}}(I^{(k-1)}), D^{2^{k-1}}(I^{(k-1)})]\subseteq
D^{2^{k}}(I^{(k)})+I.$$
\end{proof}

\begin{lemma}\label{binom}
An odd prime number $p$ does not divide the binomial coefficients
${2\choose 1}, \ {2^{2}\choose 2}, \ \ldots , {2^{n}\choose
2^{n-1}}$ if and only if  $n<\log _{2}{p}.$

\end{lemma}

\begin{proof}
Let an odd  prime number $p$ do not divide the above  mentioned
binomial coefficients  and $n\geq \log _{2}p$.  Then we have
$p<2^{n}$ and therefore the number $p$ lies  in some interval
$(2^{k-1}, 2^{k})$ for some $k <n$. It is easy to see that the
prime number $p$ does divide the binomial coefficient
${2^{k}\choose 2^{k-1}}=\frac{1\cdot 2\cdots \cdot 2^{k}}{(1\cdot
2\cdots \cdot 2^{k-1})^{2}}$ because $p$ divides the numerator and
does not divide the denominator of this fraction. This contradicts
to our assumption and therefore $n<\log _{2}p$. Let now $n<\log
_{2}p.$ Then $p>2^{n}$  and the number $p$ does not divide all the
binomial coefficients ${2\choose 1}, \ {2^{2}\choose 2}, \ \ldots
, {2^{n}\choose 2^{n-1}}$. The proof is complete.

\end{proof}

\begin{theorem}\label{solvability}
Let $L$ be a Lie algebra over a field $F$ and let  $I$ be its
solvable ideal of derived length $n$. Then the  ideal $I+D(I)$ is
solvable and its  derived length $\leq 2n$ in the following cases:
1) $\char F=0;$ 2) $n< \log _{2}p$ where $p=\char F>0.$

\end{theorem}

\begin{proof} If $p=2$ then the statement is true because
in this case $s(I)=n=0$ i.e. $I=0.$ So, we assume that $\char F=0$
or $\char F=p>2.$ Denote $J=I+D(I)$ and show that for every
positive integer $k, \ (k\leq n) $ it holds $J^{(k)}\subseteq
D^{2^{k}}(I^{(k)})+I.$ We shall prove this inclusion by induction
on $k.$ If $k=1$ then
$$ J^{(1)}=[J, J]=[I+D(I), I+D(I)]\subseteq I+[D(I), D(I)].$$
For arbitrarily chosen elements $x, y\in I$ we have by the
Leibnitz's rule
$$[D(x), D(y)]=\frac{1}{2}\{ D^{2}([x, y])-[D^{2}(x), y]-[x,
D^{2}(y)]\}.$$ Thus $[D(I), D(I)]\subseteq D^{2}(I^{(1)})+I$ and
therefore $J^{(1)}\subseteq D^{2}(I^{(1)})+I.$ Assuming that the
inclusion is proved for $k-1$ we shall prove it for $k.$ By
inductive hypothesis $J^{(k-1)}\subseteq
D^{2^{k-1}}(I^{(k-1)})+I.$
 Therefore $$J^{(k)}=[J^{(k-1)},
J^{(k-1)}]\subseteq [D^{2^{k-1}}(I^{(k-1)})+I,
D^{2^{k-1}}(I^{(k-1)})+I].$$ Since $I$ is an ideal of $L$, it
follows from above the inclusion   $$J^{(k)} \subseteq
[D^{2^{k-1}}(I^{(k-1)}), D^{2^{k-1}}(I^{(k-1)})]+I.$$ From the
conditions of Theorem and by Lemma~\ref{binom} it follows that
$\char F$ does not divide the binomial coefficient $2^{k}\choose
2^{k-1}$. Then by Lemma~\ref{key} we obtain
$$[D^{2^{k-1}}(I^{(k-1)}), D^{2^{k-1}}(I^{(k-1)})]\subseteq
D^{2^{k}}(I^{(k)})+I$$ and therefore $J^{(k)}\subseteq
D^{2^{k}}(I^{(k)})+I$. Put now in the last inclusion $k=n=s(I)$
where $s(I)$ is the derived length of $I$. Then it holds
$$J^{(n)}\subseteq D^{2^{n}}(I^{(n)})+I=D^{2^{n}}(0)+I=I.$$ But
then $J^{(2n)}\subseteq I^{(n)}=0$, i.e. $J=I+D(I)$ is a solvable
ideal of $L$ of derived length $\leq 2n.$  The proof is complete.

\end{proof}

The next Theorem is the main result of the paper. It follows
immediately from Theorem~\ref{solvability}.

\begin{theorem}\label{radical}
Let $L$ be a Lie algebra over a field $F$ and let $S(L)$ be the
sum of all solvable ideals of $L.$ Then $S(L)$ is a characteristic
ideal of the algebra $L$ in the  following cases: 1) $\char F=0;$
2) $S(L)$ is solvable of derived length $<\log _{2}p$ where
$p=\char F>0.$
\end{theorem}

 \section{On   derived length of ideals $I+D(I)+\cdots D^{k}(I)$ }
\label{}

\begin{lemma}\label{estimation}
Let $L$ be a Lie algebra over a field $F$ of characteristic $0$
and $I$ be an abelian ideal of $L.$ Then for every derivation
$D\in Der(L)$ the ideal $J_{k}=I+D(I)+\cdots +D^{k}(I)$ is
solvable of derived length $\leq k+1$ for any positive integer
$k.$
\end{lemma}

\begin{proof}
Let us show by induction on $k$ that $[D^{k}(I),
D^{k}(I)]\subseteq J_{k-1},$ where $J_{k-1}=I+D(I)+\cdots
+D^{k-1}(I).$ If $k=1$, then $[D(I), D(I)]\subseteq D^{2}([I,
I])+[D(I), I]+[I, D(I)]\subseteq I$ because $I$ is an ideal of
$L.$ In our notations $J_{0}=I+D^{0}(I)=I$ and the base of
induction is proved. Assume that it holds the inclusion
$[D^{k-1}(I), D^{k-1}(I)]\subseteq J_{k-2}.$ Take  arbitrary
elements $x, y\in I$ and consider the  equality
$$0=D^{2k}([x, y])=\sum _{s=0}^{2k}{2k\choose s}[D^{s}(x),
D^{2k-s}(y)].$$

After carrying the central summand from the right side to the left
side we obtain
$$-{2k\choose k}[D^{k}(x), D^{k}(y)]=\sum _{s=0, s\not=
k}^{2k}{2k\choose s}[D^{s}(x), D^{2k-s}(y)].
$$ Since either $s<k$ or $2k-s<k$ in the later sum, it is obvious
that all summands in the right side of this equality lie in
$J_{k-1}$ (by the induction hypothesis). But then $[D^{k}(x),
D^{k}(y)]\in J_{k-1}$, which implies the inclusion $[D^{k}(I),
D^{k}(I)]\in J_{k-1}$ because elements $x, y$ were arbitrarily
chosen. As $J_{k}=J_{k-1}+D^{k}(I)$ we obtain the inclusion
$[J_{k}, J_{k}]\subseteq J_{k-1}$, which means that the ideal
$J_{k}$ is solvable of derived length $\leq k+1$.
\end{proof}
The next statement can be easily obtained from elementary
properties of linear recurrence relations (see, for
example~\cite{handbook}, \S 3.3.3)
\begin{lemma}\label{recurrence}
  Let $ a_{n}-a_{n-1}=f(n), \ a_{0}=c, \ n\geq 1$ be a real nonhomogeneous
  first-order recurrence relation where $f(x)$ is a polynomial of degree
  $k$ and $a_{0}=c$ is an initial condition. Then this
  relation has a unique solution which is of the form $\{
  a_{n}=\phi (n)\}$ where $\phi (x)$ is a polynomial of degree
  $k+1$.

\end{lemma}

\begin{lemma}\label{evaluation}

Let $a_{0}, a_{1}, a_{2}, \ldots , a_{n}, \ldots $ be a given
sequence of real numbers that satisfy inequalities $a_{n}\leq
a_{n-1}+f(n), n=1, 2, \ldots $, where $f(x)$ is a given polynomial
of degree $k$ Then there exists a polynomial $g(x)$ of degree
$k+1$ such that $a_{n}\leq g(n)$ for all $n=0, 1, 2, \ldots $.
\end{lemma}

\begin{proof}
Consider the sequence $b_{0}, b_{1}, b_{2}, \ldots , b_{n}, \ldots
$, where $b_{0}=a_{0}, b_{1}=a_{0}+f(1), b_{n}=b_{n-1}+f(n)$ for
all $n.$ For this sequence, which obviously satisfies the
recurrence relation $b_{n}=b_{n-1}+f(n)$ with the initial
condition $b_{0}=a_{0}$, one can find by Lemma~\ref{recurrence}
the unique solution of the form $b_{n}=g(n)$ where $g(x)$ is a
polynomial of degree $k+1.$ Further, note that $a_{0}\leq b_{0}$
and by conditions of Lemma $a_{1}\leq a_{0}+f(1)=b_{1}.$ If the
inequality $a_{k-1}\leq b_{k-1}$ is proved then we have $a_{k}\leq
a_{k-1}+f(k)\leq b_{k-1}+f(k)=b_{k}$. Therefore $a_{n}\leq
b_{n}=g(n)$ for all $n.$ The proof is complete.
\end{proof}

\begin{theorem}\label{degree}
Let $L$ be a Lie algebra over a field of characteristic $0$ and
let $I$ be its solvable ideal of derived length $n.$ Then there
exists a polynomial $f_{n}(x)$ of degree $n$ such that  for every
derivation $D\in Der(L)$ the ideal $I+D(I)+\cdots +D^{k}(I)$ is
solvable of derived length $\leq f_{n}(k)$.
\end{theorem}

\begin{proof}
Induction on $n.$ If $n=1$ then by Lemma~\ref{estimation} the
derived length of the ideal $I_{k}=I+D(I)+\cdots +D^{k}(I)$ does
not exceed $k+1$, so one can take $f_{1}(x)=x+1.$ Assume that the
inequality $s(I_{k})\leq f_{n-1}(k)$ is proved for some polynomial
$f_{n-1}(x)$ of degree $n-1$ provided that $s(I)=n-1. $ Denote for
convenience by $s^{(n)}_{k}$ the derived length of the ideal
$I_{k}$, where $I$ is a solvable ideal of derived length $n.$ Take
any elements $x, y\in I$ and consider the relation
$$[D^{k}(x), D^{k}(y)]={2k\choose k}^{-1}\cdot \{ D^{2k}([x,
y])-$$
$$-\sum _{s=0, s\not= k}^{2k}{2k\choose s}[D^{s}(x), D^{2k-s}(y)]\}
\eqno (3) $$ which can be easily deduced from (1) where $m=2k$.
Since at least one of numbers $s$ or $2k-s$ is less than $k$, we
obtain that the sum $\sum _{s=0, s\not= k}^{2k}{2k\choose
s}[D^{s}(x), D^{2k-s}(y)]$ belongs to $I_{k-1}.$  The sum
$D^{2k}(I^{(1)})+I_{k-1}$ is contained in the ideal
$$I^{(1)}+D(I^{(1)})+\ldots +D^{2k}(I^{(1)})+I_{k-1}$$
whose derived length does not exceed
$s^{(n-1)}_{2k}+s^{(n)}_{k-1}.$  So, from the relation (3) we
obtain in our notations $s^{(n)}_{k}-1\leq
s^{(n)}_{k-1}+s^{(n-1)}_{2k}.$ By induction hypothesis it holds
$s^{(n-1)}_{2k}\leq f_{n-1}(2k)$ and therefore $s^{(n)}_{k}\leq
s^{(n)}_{k-1}+f_{n-1}(2k)+1.$ Denote $h_{n-1}(x)=f_{n-1}(2x)+1.$
Then $h_{n-1}(x)$ is a polynomial of degree $n-1.$ We obtain a
recurrence relation $s^{(n)}_{k}\leq s^{(n)}_{k-1}+h_{n-1}(k).$ By
Lemma~\ref{evaluation} it holds $s^{(n)}_{k}\leq f_{n}(k)$, where
$f_{n}(x)$ is a polynomial of degree $n.$  The proof is complete.

\end{proof}

\end{document}